\documentclass[a4paper]{article}
\usepackage{amsthm,amssymb,amsmath,enumerate}
\usepackage{algorithm}

\usepackage{cite}

\usepackage{xcolor}

\usepackage{tikz}
\usetikzlibrary{calc}
\usetikzlibrary{backgrounds}

\colorlet{hellgrau}{black!20!white}
\colorlet{dunkelgrau}{black!60!white}

\colorlet{grau}{black!40!white}
\colorlet{bold}{black} 
\tikzstyle{ledge}=[thick, grau]
\tikzstyle{rededge}=[very thick, bold]
\tikzstyle{lvertex}=[thick,circle,inner sep=0.cm, minimum size=2mm, fill=white, draw=grau]
\tikzstyle{redvx}=[thick,circle,inner sep=0.cm, minimum size=2mm, fill=white, draw=bold]

\tikzstyle{hvertex}=[thick,circle,inner sep=0.cm, minimum size=2mm, fill=white, draw=black]
\tikzstyle{hedge}=[very thick]
\tikzstyle{medge}=[thick]
\tikzstyle{harrow}=[thick,arrows=->]
\tikzstyle{darrow}=[thick,arrows=<-]
\tikzstyle{point}=[draw,circle,inner sep=0.cm, minimum size=1mm, fill=black]
\tikzstyle{pointer}=[thick,->,shorten >=2pt,color=dunkelgrau]
\tikzstyle{facebdry}=[color=auchblau, very thick] 
\tikzstyle{face}=[facebdry,fill=hellblau]
\tikzstyle{nface}=[color=hellblau,fill=hellblau,thick] 
\tikzset{>={latex}}
\tikzstyle{tinyvx}=[thick,circle,inner sep=0.cm, minimum size=1.3mm, fill=white, draw=black]
\tikzstyle{smallvx}=[hvertex,minimum size=1.7mm]

\pgfdeclarelayer{background}
\pgfdeclarelayer{foreground}
\pgfsetlayers{background, main, foreground}

\DeclareMathOperator*{\intr}{int}

\newtheorem*{theo*wallsize}{Theorem~\ref{treew:theo:theoremAsWallSize}}
\newtheorem*{lemma*linkage}{Lemma~\ref{treew:lem:linkageMainLemma}}

\newcommand{\comment}[1]{}
\newcommand{\N}{{\mathbb Z}_+}

\newcommand{\emtext}[1]{\text{\em #1}}

\newcommand{\sm}{\setminus}

\title{Subcubic graphs of large treewidth do not have the edge-Erd\H{o}s-P\'{o}sa property}
\author{Henning Bruhn and Raphael Steck}
\date{\today}

\usepackage{hyperref}
\hypersetup{
pdftitle={Subcubic graphs of large treewidth do not have the edge-Erdös-P\'{o}sa property}, 
pdfauthor={Raphael Steck},
pdfsubject={Edge-Erdös-P\'{o}sa property},
pdfproducer={pdfeTex 3.14159-1.30.6-2.2},
colorlinks=false,
pdfborder=0 0 0	
}

\usepackage{enumitem}

\theoremstyle{plain}
\newtheorem{theo}{Theorem}
\newtheorem{lemma}[theo]{Lemma}

\begin{document}

\maketitle

\begin{abstract}
We show that subcubic graphs of treewidth at least $2500$ do not have the edge-Erd\H{o}s-P\'{o}sa property.
\end{abstract}



\section{Introduction}

Menger's theorem provides a strong duality between packing and covering for paths: In every graph $G$, there are either $k$ disjoint paths between predefined sets $A, B \subseteq V(G)$, or there is a set $X \subseteq V(G)$ of size at most $k$ such that $G - X$ contains no $A$--$B$~path. Relaxed versions of this result exist for many sets of graphs, and we call this duality the \emph{Erd\H{o}s-P\'{o}sa property}.
In this article, we focus on the edge variant: A class $\mathcal{F}$ has the \emph{edge-Erd\H{o}s-P\'{o}sa property} if there exists a function $f: \N \rightarrow \mathbb{R}$ such that for every graph $G$ and every integer $k$, there are $k$ edge-disjoint subgraphs of $G$ each isomorphic to some graph in $\mathcal{F}$ or there is an edge set $X \subseteq E(G)$ of size at most $f(k)$ meeting all subgraphs of $G$ isomorphic to some graph in $\mathcal{F}$. The edge set $X$ is called the \emph{hitting set}.
If we replace vertices with edges in the above definition, that is, if we look for a vertex hitting set or vertex-disjoint graphs, then we obtain the \emph{vertex-Erd\H{o}s-P\'{o}sa property}.
The class $\mathcal{F}$ that is studied in this article arises from taking minors: For a fixed graph $H$, we define the set 
\(
\mathcal{F}_H = \{ G \, : \, H \text{ is a minor of } G\}.
\)
Any graph $G \in \mathcal{F}_H$ is called an \emph{$H$-expansion}.

The vertex-Erd\H{o}s-P\'{o}sa property for $\mathcal{F}_H$ is well understood:
Robertson and Seymour \cite{robertson86} proved that the class $\mathcal{F}_H$ has the vertex-Erd\H{o}s-P\'{o}sa property if and only if $H$ is planar. While both the vertex- and the edge-Erd\H{o}s-P\'{o}sa property are false for all non-planar graphs $H$ (see for example \cite{raymond17}), the situation is much more mysterious for planar graphs. For some simple planar graphs $H$ such as long cycles\cite{BHJ19} or $K_4$\cite{BH18}, $\mathcal{F}_H$ still has the edge-Erd\H{o}s-P\'{o}sa property, while for some others, for example subcubic trees of large pathwidth\cite{bruhn18}, it does not. For most planar graphs, it is unknown whether the edge-Erd\H{o}s-P\'{o}sa property holds or not.
For an overview of results on the Erd\H{o}s-P\'{o}sa-property, we recommend the website of Jean-Florent Raymond \cite{raymondweb}.

We partially fill this gap by proving that for every subcubic graph of large treewidth $H$, $\mathcal{F}_H$ does not have the edge-Erd\H{o}s-P\'{o}sa property. Note that while it was known that large walls do not have the edge-Erd\H{o}s-P\'{o}sa property (claimed without proof in \cite{bruhn18}), this does not imply our main result as, unlike the vertex-Erd\H{o}s-P\'{o}sa property, is not known whether the edge variant is closed under taking minors.

\begin{theo} \label{treew:thm:noEEPwhenLargeWallContained}
For subcubic graphs $H$ of treewidth at least $2500$, $\mathcal{F}_H$ does not have the edge-Erd\H{o}s-P\'{o}sa property.
\end{theo}

To prove Theorem~\ref{treew:thm:noEEPwhenLargeWallContained}, we only use treewidth to deduce that $H$ contains a large wall, for which we use the linear bound provided by Grigoriev \cite{grigoriev}. So in fact, we show the following theorem:

\begin{theo} \label{treew:theo:theoremAsWallSize}
For subcubic graphs $H$ that contain a wall of size $250\times 250$, $\mathcal{F}_H$ does not have the edge-Erd\H{o}s-P\'{o}sa property.
\end{theo} 

There is room for improvement in the theorem. Requiring the graph $H$ to be subcubic simplifies the argument considerably, but we suspect it is not necessary. Moreover, we believe that with a more careful but somewhat tedious analysis the wall size could be dropped to about $30\times 30$. Still, this seems unlikely to be close to be best possible. Indeed, walls of size $6\times 4$ do not have the edge-Erd\H{o}s-P\'{o}sa property~\cite{steck23wallserdosposa}. (Whether  graphs containing  $6\times 4$-walls have the property is not known.)

\section{Construction} \label{sec:treew:construction}

There is only one known tool to prove that a set $\mathcal{F}_H$ of $H$-expansions that satisfies the vertex-Erd\H{o}s-P\'{o}sa property does not have the edge-Erd\H{o}s-P\'{o}sa property: The \emph{Heinlein Wall}, after \cite{bruhn18}, shown at size~5 in Figure~\ref{treew:fig:condWall}.

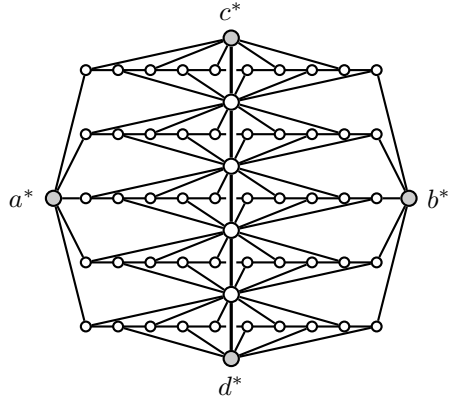
\begin{figure}[bht] 
\centering
\begin{tikzpicture}[scale=0.85]
\tikzstyle{tinyvx}=[thick,circle,inner sep=0.cm, minimum size=1.3mm, fill=white, draw=black]

\def\vstep{1}
\def\hstep{0.5}
\def\hwidth{9}
\def\hheight{4}

\def\totalheight{\hheight*\vstep}
\def\totalwidth{\hwidth*\hstep}
\pgfmathtruncatemacro{\minustwo}{\hwidth-2}
\pgfmathtruncatemacro{\minusone}{\hwidth-1}

\foreach \j in {0,...,\hheight} {
\draw[medge] (0,\j*\vstep) -- (\hwidth*\hstep,\j*\vstep);
\foreach \i in {0,...,\hwidth} {
\node[tinyvx] (v\i\j) at (\i*\hstep,\j*\vstep){};
}
}

\foreach \j in {1,...,\hheight}{
\node[hvertex] (z\j) at (0.5*\hwidth*\hstep,\j*\vstep-0.5*\vstep) {};
}
\pgfmathtruncatemacro{\plusvone}{\hheight+1}

\node[hvertex,fill=hellgrau,label=above:$c^*$] (z\plusvone) at (0.5*\totalwidth,\totalheight+0.5*\vstep) {};
\node[hvertex,fill=hellgrau,label=below:$d^*$] (z0) at (0.5*\totalwidth,-0.5*\vstep) {};

\foreach \j in {1,...,\plusvone}{
\pgfmathtruncatemacro{\subone}{\j-1}
\draw[line width=1.3pt,double distance=1.2pt,draw=white,double=black] (z\j) to (z\subone);
\foreach \i in {0,2,...,\hwidth}{
\draw[medge] (z\j) to (v\i\subone);
}
}

\foreach \j in {0,...,\hheight}{
\foreach \i in {1,3,...,\hwidth}{
\draw[medge] (z\j) to (v\i\j);
}
}

\pgfmathtruncatemacro{\minusvone}{\hheight-1}
\node[hvertex,fill=hellgrau,label=left:$a^*$] (a) at (-\hstep,0.5*\totalheight) {};
\foreach \j in {0,...,\hheight} {
\draw[medge] (a) -- (v0\j);
}

\node[hvertex,fill=hellgrau,label=right:$b^*$] (b) at (\totalwidth+\hstep,0.5*\totalheight) {};
\foreach \j in {0,...,\hheight} {
\draw[medge] (v\hwidth\j) to (b);
}
\end{tikzpicture}
\caption{A Heinlein Wall of size 5.}
\label{treew:fig:condWall}
\end{figure}

For any integer $n \in \N$, we define $[n] = \{1, \ldots, n\}$.
A Heinlein Wall $W$ of size $r \in \N$ is the graph consisting of the following:
\begin{itemize}
\item For every $j \in [r]$, let $P^j = u^j_1 \ldots u^j_{2r}$ be a path of length $2r - 1$ and for $j \in \{0\} \cup [r]$, let $z_j$ be a vertex. Moreover, let $a^*$, $b^*$ be two further vertices.
\item For every $i, j \in [r]$, add the edges $z_{j-1} u^j_{2i - 1}, z_{j} u^j_{2i}, z_{i-1} {z_i}, a^* u^j_{1}$ and $b^* u^j_{2r}$.
\end{itemize}
We define $c^* = z_0$ and $d^* = z_r$. 
We call the vertices $a^*,b^*,c^*$ and $d^*$ \emph{terminals} of $W$, while 
the vertices $z_j, j \in \{0\} \cup [r]$ are called \emph{bottleneck vertices}.
Additionally, we define $W^0 = W - \{a^*,b^*,c^*,d^*\}$.

An \emph{($a^*$--$b^*$, $c^*$--$d^*$) linkage} is the vertex-disjoint union of an $a^*$--$b^*$~path with a $c^*$--$d^*$~path.
We need an easy observation:
\begin{lemma}[Bruhn et al \cite{bruhn18}] \label{lem:noTwoLinkages}
There are no two edge-disjoint ($a^*$--$b^*$, $c^*$--$d^*$) linkages in a Heinlein Wall.
\end{lemma}

For $m,n \in \N$, an \emph{elementary grid} of size $m \times n$ is a graph with vertices $v_{i,j}$ for all $i \in [m], j \in [n]$ and edges $v_{i,j} v_{i+1,j} \,\, \forall i \in [m-1], j \in [n]$ as well as $v_{i,j} v_{i,j+1} \,\, \forall i \in [m], j \in [n-1]$. A \emph{grid} is a subdivision of an elementary grid.

A wall is the subcubic variant of a grid. We define an elementary wall as an elementary grid with every second vertical edge removed. That is, an elementary wall of size $m \times n$ is an elementary grid of size $(m+1) \times (2n+2)$ with every edge $v_{i,2j} v_{i+1,2j} \, , i \in [m], i \text{ is odd}, j \in [n+1]$ and every edge $v_{i,2j-1} v_{i+1,2j-1} \, , i \in [m], i \text{ is even}, j \in [n+1]$ being removed. Additionally, we remove all vertices of degree~$1$ and their incident edges.
The \emph{$i^\text{th}$ row} of an elementary wall is the induced subgraph on $v_{i,1},\ldots, v_{i,2n+2}$ for $i\in [m+1]$ (ignore the vertices that have been removed); this is a path. 
There is a set of exactly $n+1$ disjoint paths between the first row and the $(m+1)^\text{th}$ row. These paths are the \emph{columns} of an elementary wall. The \emph{bricks} of an elementary wall are its $6$-cycles. (See Figure~\ref{fig:elwall})

\begin{figure}[htb] 
\centering
\begin{tikzpicture}
\tikzstyle{tinyvx}=[thick,circle,inner sep=0.cm, minimum size=1.3mm, fill=white, draw=black]
\tikzstyle{vx}=[thick,circle,inner sep=0.cm, minimum size=1.6mm, fill=white, draw=black]
\tikzstyle{marked}=[line width=3pt,color=dunkelgrau]
\tikzstyle{point}=[thin,->,shorten >=2pt,color=dunkelgrau]
\tikzstyle{edg}=[draw,thick]

\def\wallheight{8}
\def\brickheight{0.4}

\pgfmathtruncatemacro{\lastrow}{\wallheight}
\pgfmathtruncatemacro{\penultimaterow}{\wallheight-1}
\pgfmathtruncatemacro{\lastrowshift}{mod(\wallheight,2)}
\pgfmathtruncatemacro{\lastx}{2*\wallheight+1}

\draw[edg] (\brickheight,0) -- (2*\wallheight*\brickheight+\brickheight,0);
\foreach \i in {1,...,\penultimaterow}{
  \draw[edg] (0,\i*\brickheight) -- (2*\wallheight*\brickheight+\brickheight,\i*\brickheight);
}
\draw[edg] (\lastrowshift*\brickheight,\lastrow*\brickheight) to ++(2*\wallheight*\brickheight,0);

\foreach \j in {0,2,...,\penultimaterow}{
  \foreach \i in {0,...,\wallheight}{
    \draw[edg] (2*\i*\brickheight+\brickheight,\j*\brickheight) to ++(0,\brickheight);
  }
}
\foreach \j in {1,3,...,\penultimaterow}{
  \foreach \i in {0,...,\wallheight}{
    \draw[edg] (2*\i*\brickheight,\j*\brickheight) to ++(0,\brickheight);
  }
}

\def\colind{5}
\foreach \j in {2,4,6}{
  \draw[marked] (\colind*\brickheight,\j*\brickheight-2*\brickheight) -- ++ (0,\brickheight) -- ++(-\brickheight,0) -- ++(0,\brickheight) -- ++(\brickheight,0);
}
\draw[marked] (\colind*\brickheight,6*\brickheight) -- ++ (0,\brickheight) -- ++(-\brickheight,0) -- ++(0,\brickheight);

\def\rowind{4}
\foreach \i in {1,...,\lastx}{
  \draw[marked] (\i*\brickheight-\brickheight,\rowind*\brickheight) -- ++(\brickheight,0);
}

\draw[marked] (2*\wallheight*\brickheight,1*\brickheight) -- ++(0,\brickheight) coordinate[midway] (brx)
-- ++(-2*\brickheight,0)
-- ++(0,-\brickheight) -- ++(2*\brickheight,0);

\foreach \i in {1,...,\lastx}{
  \node[tinyvx] (w\i w0) at (\i*\brickheight,0){};
}
\foreach \j in {1,...,\penultimaterow}{
  \foreach \i in {0,...,\lastx}{
    \node[tinyvx] (w\i w\j) at (\i*\brickheight,\j*\brickheight){};
  }
}
\foreach \i in {1,...,\lastx}{
  \node[tinyvx] (w\i w\lastrow) at (\i*\brickheight+\lastrowshift*\brickheight-\brickheight,\lastrow*\brickheight){};
}

\foreach \i in {2,4,...,\lastx}{
  \node[tinyvx,fill=white] (w\i w\lastrow) at (\i*\brickheight+\lastrowshift*\brickheight-\brickheight,\lastrow*\brickheight){};
}

\node[anchor=mid] (tr) at (\lastx*\brickheight+0.5,\wallheight*\brickheight+0.8){$1$\textsuperscript{st} row};
\draw[point,out=270,in=0] (tr) to (w\lastx w\wallheight);


\node[anchor=mid] (vp) at (0,\wallheight*\brickheight+0.8){column};
\draw[point,out=0,in=90] (vp) to (w\colind w\wallheight);

\node[align=center] (hp) at (\lastx*\brickheight+1.2,\rowind*\brickheight+0.8){row};
\draw[point,out=270,in=0] (hp) to (w\lastx w\rowind);

\node[align=center] (br) at (\lastx*\brickheight+1.2,1*\brickheight+0.8){brick};
\draw[point,out=270,in=0] (br) to (brx);

\end{tikzpicture}
\caption{An elementary wall of size $8\times8$.}\label{fig:elwall}
\end{figure}
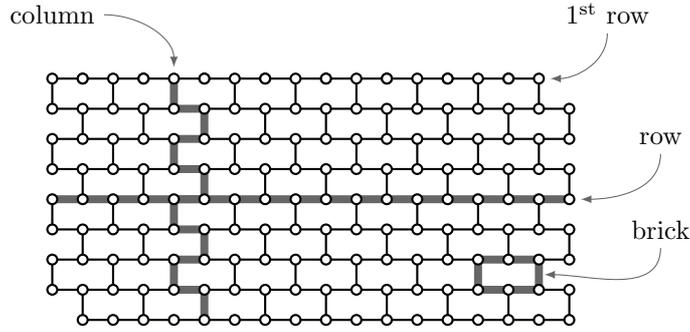

A \emph{wall} is defined as the subdivision of an elementary wall. However, elementary walls have some vertices of degree~$2$ on the outer face of the wall. As we never want to distinguish between graphs that only differ by subdivision of edges, we avoid some annoying technicalities by slightly modifying the above definition. We define a \emph{wall'} of size $m \times n$ as the subdivision of an elementary wall of size $m \times n$ with all degree~$2$ vertices being contracted. (See Figure~\ref{fig:WallvsWall'})

\begin{figure}[hbt] 
\centering
\begin{tikzpicture}[scale=0.5]
\tikzstyle{tinyvx}=[thick,circle,inner sep=0.cm, minimum size=1.3mm, fill=white, draw=black]

\def\hstep{1}
\def\vstep{1}
\def\width{7}
\def\height{4}

\def\totalheight{\height*\vstep}
\pgfmathtruncatemacro{\wminusone}{\width-1}
\pgfmathtruncatemacro{\hminusone}{\height-1}

\begin{scope}[shift={(-10,0)}]

\foreach \i in {1,...,\width} {
	\node[tinyvx] (v\i0) at (\i*\hstep,0){};
}
\foreach \j in {1,...,\hminusone} {
	\foreach \i in {0,...,\width} {
		\node[tinyvx] (v\i\j) at (\i*\hstep,\j*\vstep){};
	}
}
\foreach \i in {0,...,\wminusone} {
	\node[tinyvx] (v\i\height) at (\i*\hstep,\totalheight){};
}

\foreach \i in {2,...,\width} {
	\pgfmathtruncatemacro{\subone}{\i-1}
	\draw[medge] (v\subone0) -- (v\i0);
}
\foreach \j in {1,...,\hminusone} {
	\foreach \i in {1,...,\width} {
		\pgfmathtruncatemacro{\subone}{\i-1}
		\draw[medge] (v\subone\j) -- (v\i\j);
	}
}
\foreach \i in {1,...,\wminusone} {
	\pgfmathtruncatemacro{\subone}{\i-1}
	\draw[medge] (v\subone\height) -- (v\i\height);
}

\foreach \j in {0,2} {
	\foreach \i in {1,3,5,7} {
		\pgfmathtruncatemacro{\plusone}{\j+1}
		\draw[medge] (v\i\plusone) -- (v\i\j);
	}
}
\foreach \j in {1,3} {
	\foreach \i in {0,2,4,6} {
		\pgfmathtruncatemacro{\plusone}{\j+1}
		\draw[medge] (v\i\plusone) -- (v\i\j);
	}
}
\end{scope}

\foreach \i in {3,5} {
	\node[tinyvx] (v\i0) at (\i*\hstep,0){};
}
\foreach \j in {1,...,\hminusone} {
	\foreach \i in {1,...,\wminusone} {
		\node[tinyvx] (v\i\j) at (\i*\hstep,\j*\vstep){};
	}
}
\foreach \i in {2,4} {
	\node[tinyvx] (v\i\height) at (\i*\hstep,\totalheight){};
}

\draw[medge] (v30) -- (v50);
\foreach \j in {1,2,3} {
	\foreach \i in {2,...,\wminusone} {
		\pgfmathtruncatemacro{\subone}{\i-1}
		\draw[medge] (v\subone\j) -- (v\i\j);
	}
}
\draw[medge] (v24) -- (v44);

\foreach \j in {0,2} {
	\foreach \i in {3,5} {
		\pgfmathtruncatemacro{\plusone}{\j+1}
		\draw[medge] (v\i\plusone) -- (v\i\j);
	}
}
\foreach \j in {1,3} {
	\foreach \i in {2,4} {
		\pgfmathtruncatemacro{\plusone}{\j+1}
		\draw[medge] (v\i\plusone) -- (v\i\j);
	}
}
\draw[medge] (v61) -- (v62);
\draw[medge] (v12) -- (v13);
\draw[medge] (v13) to [out=180,in=270] (0.2,3.5) to [out=90,in=180] (v24);
\draw[medge] (v61) to [out=0,in=90] (6.8,0.5) to [out=270,in=0] (v50);
\draw[medge] (v11) to [out=180,in=270] (0.2,1.5) to [out=90,in=180] (v12);
\draw[medge] (v62) to [out=0,in=270] (6.8,2.5) to [out=90,in=0] (v63);
\draw[medge] (v30) to [out=180,in=270] (v11);
\draw[medge] (v44) to [out=0,in=90] (v63);


\end{tikzpicture}
\caption{An elementary wall of size $4 \times 3$ and a wall' of the same size.}
\label{fig:WallvsWall'}
\end{figure}
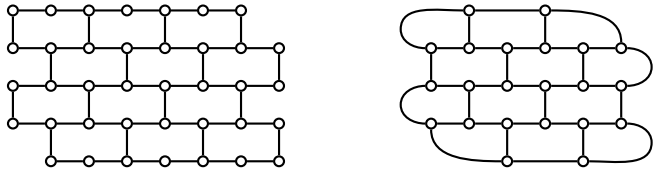

Throughout, we will use this slightly modified definition of a wall.
The key properties of a wall, such as large treewidth and planarity, carry over to a wall'. 
The definition of rows, columns and bricks in an elementary wall carries over to a wall' in a natural way (with some truncation of the first and last row and column). For brevity of notation, we define an \emph{$n$-wall'} as a wall' of size~$n\times n$. 

The \emph{outercycle} of a wall' $W$ is the cycle $C$ contained in $W$ that contains the first and last row and first and last column. 
Two vertices $u,v$ of $W$ are \emph{$d$-apart in $W$} if every $u$--$v$~path in $W$, every $u$--$C$~path and every $v$--$C$~path in $W$ intersects at least~$d+1$ rows or at least $d+1$ columns of $W$.
We extend the definition to bricks by saying that two bricks $B_1,B_2$ of $W$ are \emph{$d$-apart in $W$}
if every pair of one vertex from $B_1$ and one vertex from $B_2$ is $d$-apart in $W$. Note that if $v_1,v_2$ are $d$-apart and if $v_1$ lies in the brick $B_1$, and $v_2$ in the brick $B_2$ then $B_1,B_2$ are $(d-2)$-apart.

Note, furthermore, that if $W$ is part of a planar graph $G$ then there are no shortcuts in $G$. That is, if $u,v$ are $d$-apart in $W$ then there is also no $u$--$v$~path \emph{in} $G$ that meets fewer than $d+1$ rows and columns of $W$, and the same holds true for paths from $u$ or $v$ to the outercycle.

To apply Menger's theorem, for $n \in \N$ and vertex sets $A$ and $B$ in a graph $G$, we define an \emph{$n$-separator} as a vertex set $X \subseteq V(G)$ of size $|X| \leq n$ such that there is no $A$--$B$~path in $G-X$. We will usually apply this for one side being a single vertex, that is $A = \{a\}$, in which case we additionally require that $a \not\in X$.

\section{Large treewidth results} \label{sec:subcubicResults}

How do we prove our main result? Let $H$ be a planar subcubic graph of treewidth $\geq 2500$.
Given a size $r$ of a hypothetical hitting set, we show that there is a graph $Z$ that neither contains two edge-disjoint subdivisions of $H$, nor admits an edge set $U$ of size $|U|\leq r$ such that $Z-U$ is devoid of subdivisions of $H$.
That then proves that $\mathcal{F}_H$ does not have the edge-Erd\H os-P\'osa property.

Since $H$ has treewidth $\geq 2500$, it contains a grid-minor of size at least $501 \times 501$ \cite{grigoriev} and thus a wall' $M$ of size at least $250 \times 250$. 
We pick two edges $e_1$ and $e_2$ of $M$ such that both of them are incident with a branch vertices of degree~$3$ of $M$ and such that
\begin{equation}\label{e1e2distance}
\begin{minipage}[c]{0.8\textwidth}\em
every pair of one endvertex from $e_1$ and one endvertex from $e_2$ is $70$-apart in $M$.
\end{minipage}\ignorespacesafterend 
\end{equation} 
As $H$ is planar and $M$ large enough it is possible to find such edges $e_1,e_2$.
We denote the endvertex of $e_1$ that is also a branch vertices of degree~$3$ of $M$ by $a$, and the other endvertex by $b$ (which may, or not, be a branch vertex, too). 
For $e_2$, we call its endvertices $c$ and $d$, where $c$ is chosen to be a branch vertex of degree~$3$ of $M$.

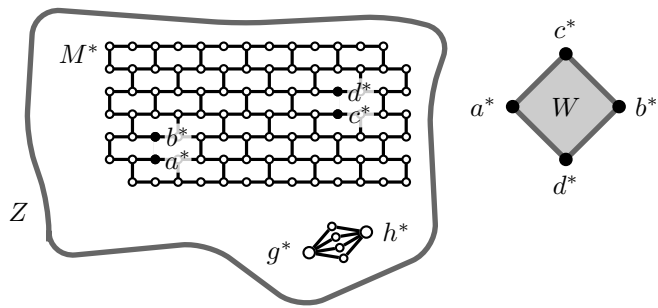
\begin{figure}[htb] 
\centering
\begin{tikzpicture}
\tikzstyle{smallvx}=[thick,circle,inner sep=0.cm, minimum size=1mm, fill=white, draw=black]
\tikzstyle{vx}=[thick,circle,inner sep=0.cm, minimum size=1.5mm, fill=white, draw=black,fill=black]

\def\wallheight{6}
\def\brickheight{0.3}

\pgfmathtruncatemacro{\lastrow}{\wallheight}
\pgfmathtruncatemacro{\penultimaterow}{\wallheight-1}
\pgfmathtruncatemacro{\lastrowshift}{mod(\wallheight,2)}
\pgfmathtruncatemacro{\lastx}{2*\wallheight+1}

\draw[hedge] (\brickheight,0) -- (2*\wallheight*\brickheight+\brickheight,0);
\foreach \i in {1,...,\penultimaterow}{
  \draw[hedge] (0,\i*\brickheight) -- (2*\wallheight*\brickheight+\brickheight,\i*\brickheight);
}
\draw[hedge] (\lastrowshift*\brickheight,\lastrow*\brickheight) to ++(2*\wallheight*\brickheight,0);

\foreach \j in {0,2,...,\penultimaterow}{
  \foreach \i in {0,...,\wallheight}{
    \draw[hedge] (2*\i*\brickheight+\brickheight,\j*\brickheight) to ++(0,\brickheight);
  }
}
\foreach \j in {1,3,...,\penultimaterow}{
  \foreach \i in {0,...,\wallheight}{
    \draw[hedge] (2*\i*\brickheight,\j*\brickheight) to ++(0,\brickheight);
  }
}

\foreach \i in {1,...,\lastx}{
  \node[smallvx] (w\i w0) at (\i*\brickheight,0){};
}
\foreach \j in {1,...,\penultimaterow}{
  \foreach \i in {0,...,\lastx}{
    \node[smallvx] (w\i w\j) at (\i*\brickheight,\j*\brickheight){};
  }
}
\foreach \i in {1,...,\lastx}{
  \node[smallvx] (w\i w\lastrow) at (\i*\brickheight+\lastrowshift*\brickheight-\brickheight,\lastrow*\brickheight){};
}

\draw[ultra thick, white] (w2w1) to (w2w2);

\draw[ultra thick, white] (w10w3) to (w10w4);

\begin{scope}[opacity=0.8]
\fill[white] (w3w1.center) circle [radius=5pt];
\fill[white] (w3w2.center) circle [radius=5pt];
\fill[white] (w11w3.center) circle [radius=5pt];
\fill[white] (w11w4.center) circle [radius=5pt];
\end{scope}
\node at (w3w1.center) {$a^*$};
\node[smallvx,fill=black] at (w2w1) {};

\node at (w3w2.center) {$b^*$};
\node[smallvx,fill=black] at (w2w2) {};

\node at (w11w3.center) {$c^*$};
\node[smallvx,fill=black] at (w10w3) {};

\node at (w11w4.center) {$d^*$};
\node[smallvx,fill=black] at (w10w4) {};

\node at (-0.4,1.7) {$M^*$};

\begin{scope}[shift={(6,1)}]
\def\size{0.7}
\draw[line width=2pt,dunkelgrau,fill=hellgrau] (-\size,0) to (0,-\size) to (\size,0) to (0,\size) to cycle;
\node[vx,label=left:$a^*$] (a) at (-\size,0){};
\node[vx,label=right:$b^*$] (b) at (\size,0){};
\node[vx,label=above:$c^*$] (c) at (0,\size){};
\node[vx,label=below:$d^*$] (d) at (0,-\size){};
\node at (0,0) {$W$};
\end{scope}

\begin{scope}[shift={(3,-0.8)},rotate=20]
\tikzstyle{vvx}=[thick,circle,inner sep=0.cm, minimum size=1.5mm, fill=white, draw=black]

\def\blubb{0.4}
\node[vvx,label=left:$g^*$] (g) at (-\blubb,0){};

\node[vvx,label=right:$h^*$] (h) at (\blubb,0){};

\def\blabb{0.15}
\foreach \i in {0,...,3}{
  \node[smallvx] (xx\i) at (0,\i*\blabb-1.5*\blabb) {};
  \draw[hedge] (g) to (xx\i) to (h);
}
\end{scope}

\tikzstyle{border}=[line width=2pt, dunkelgrau, rounded corners=10pt]

\draw[border] (-0.8,-1) to (1.3,-0.8) to (2.5,-1.7) to (4.3,-0.9) to (4.2,1.1) to (4.5,2)
to (3.7,2.2) to (1.7,2.1) to (-1,2.3) to (-1.1,0.6) 
to (-0.8,-0.3) to cycle;

\node at (-1.2,-0.4) {$Z$};

\end{tikzpicture}
\caption{Construction of the counterexample graph $Z$}
\label{fig:constructionOfZ}
\end{figure}

Given a positive integer $r$, we define $Z$ as follows:
\begin{itemize}
\item start with a copy of $H-\{e_1,e_2\}$, where we denote the copy of a vertex $h$ of $H$ by $h^*$;
\item replace every edge $g^*h^*$ in the copy of $H-\{e_1,e_2\}$ by $2r$ internally disjoint $g^*$--$h^*$~paths of length~$2$; and
\item add a Heinlein wall $W$ of size~$2r$, where the terminals $a^*,b^*$ of $W$ are identified with the endvertices of $e_1$, and where the terminals $c^*,d^*$ are identified with the endvertices of $e_2$.
\end{itemize}
A depiction of $Z$ can be seen in Figure~\ref{fig:constructionOfZ}.

We extend the mapping $V(H)\to V(Z)$ defined by $h\mapsto h^*$ to sets of vertices 
in  $H-\{e_1,e_2\}$:
for a vertex set $J\subseteq V(H)$, we set $J^*=\{h^*:h\in J\}$.

To better to distinguish between $H$ and $Z$,
we use the first half of the alphabet ($a$--$m$) for vertices, vertex sets and graphs that are part of $H$, while the second half of the alphabet ($o$--$z$) is reserved for objects belonging to $Z$.
Starred letters of the first half ($a^*$--$m^*$) are used for vertices and objects in $Z$ that have counterparts in $H$. 
We define $M^*$ to be an arbitrary subdivision of $M - \{e_1, e_2\}$ in $Z$ such that the set of its branch vertices is precisely $(V(M))^*$ and such that each subdivided edge of $M^*$ consists of one of the  $2r$ paths originating from multiplying the corresponding edge of $M - \{e_1, e_2\}$. Note that $M^*$ is a wall' except for $e_1, e_2$, and note that $M^*$ is disjoint from $W^0$.

Let us first prove the first half of Theorem~\ref{treew:thm:noEEPwhenLargeWallContained}:
there is no small edge hitting set in $Z$.
\begin{lemma}\label{nosmallhittingsetlem}
For every edge set $U$ in $Z$ of size $|U|\leq r$, the graph $Z-U$ contains a subdivision of $H$.
\end{lemma}
\begin{proof}
As for every edge $gh\in E(H)\sm \{e_1,e_2\}$, the vertices $g^*$ and $h^*$ are linked by $2r$ internally disjoint paths, we may easily find a subdivision of $H-\{e_1,e_2\}$ in $Z-U$. Moreover, $U$ is too small to meet all ($a^*$--$b^*$, $c^*$--$d^*$) linkages in the Heinlein wall $W$. Thus, the subdivision of $H-\{e_1,e_2\}$ can be extended to one of $H$ in $Z-U$.
\end{proof}
The harder part of Theorem~\ref{treew:thm:noEEPwhenLargeWallContained} is to prove that there can be no two edge-disjoint subdivisions of $H$ in $Z$. 
We will prove:
\begin{lemma} \label{treew:lem:linkageMainLemma}
Every subdivision of $H$ in $Z$ contains an ($a^*$--$b^*$, $c^*$--$d^*$) linkage in~$W$.
\end{lemma}
Recall that, by Lemma~\ref{lem:noTwoLinkages}, any two such linkages share an edge. Thus, once we have shown the above lemma, we then have finished the proof of the theorem.

When we talk about a subdivision of $H$ in $Z$, we implicitly assume that an \emph{embedding} of $H$ into $Z$ is fixed: 
a function $\Phi$ that maps every vertex of $H$ to the corresponding branch vertex in $Z$, and that maps every edge of $H$  to the corresponding subdivided edge in $Z$. We will extend such an embedding $\Phi$ to subgraphs of $H$ in the obvious way. In particular, $\Phi(H)$ then denotes the subdivision of $H$ in $Z$.

For the remainder of this article, we assume $\Phi$ to be a fixed embedding of $H$ in $Z$. We will prove Lemma~\ref{treew:lem:linkageMainLemma} for this fixed embedding of $H$.
The main difficulty is that we do not know how $H$ embeds in $Z$.
In order to get some control on what is mapped where by $\Phi$, we concentrate on a set of vertices that are well connected to large walls'. 
We will later see that only a small number of them can be mapped into $W$.
We define a \emph{$3$-fan} from a vertex $v$ to a set $S$ as the union of three non-trivial paths from $v$ to $S$ that are disjoint except for their first vertex $v$.
Set
\begin{align*}
B =  \{& h\in V(H) \, : \, \text{there is $10$-wall' $M'$ and}\\& \text{a $3$-fan from $h$ to the branch vertices of degree~$3$ of $M'$}\}.
\end{align*}
Note that $M$ is also a $10$-wall'. 

\begin{lemma}\label{nofitlem}
Not all branch vertices of any $10$-wall' can be contained in $W$.
\end{lemma}
\begin{proof}
It is easy to check that a Heinlein Wall has pathwidth at most~$5$, and thus also treewidth at most $5$. Therefore, it cannot contain a $10$-wall' since the latter has treewidth at least~$10$.
\end{proof}

As we are only ever interested in branch vertices of degree~$3$, we will call those \emph{proper branch vertices}. 
Moreover, a \emph{proper branch vertex} of $M^*$ is the image under the $*$-map of a proper branch vertex of $M$.
Note that every proper branch vertex of every $10$-wall' $M'$ is in $B$: Indeed, every proper branch vertex in $M'$ is connected to its three adjacent proper branch vertices of $M'$, and those paths form the desired $3$-fan. 
In particular, this implies that every proper branch vertex of $M$ is in $B$. Recall that by choice of $e_1$ and $e_2$, this includes $a$ and $c$. For $b$ and $d$, we do not know, but the following lemma helps to deal with them.

\begin{lemma}\label{treew:lem:bd3fan}
Let $h^* \in V(Z - W)$ and let $T\subseteq Z$ be a $3$-fan from $h^*$ to 
the union of the proper branch vertices of $M^*$ with $\{b^*,d^*\}$. 
Then there is also a $3$-fan from $h$ to proper branch vertices of $M$ in $H$.
\end{lemma}

\begin{proof}
To prove the lemma, we need to show two things: First, we need to find a $3$-fan that is disjoint from $W^0$ so we can pull it back to $H$. Second, we need to get rid of $b$ and $d$ and find a $3$-fan that connects $h$ with proper branch vertices of $M$ only.

Since the terminals $a^*$ and $c^*$ of $W$ are proper branch vertices of $M^*$ and since $h^* \not\in V(W)$, we can shorten the $3$-fan $T$ to obtain a $3$-fan that is disjoint from $W^0$ but still connects $h^*$ with proper branch vertices of $M^*$ or $b^*$ or $d^*$ if necessary.
Since this $3$-fan is disjoint from $W^0$, we can find a corresponding $3$-fan $F$ in $H$ that connects $h$ with proper branch vertices of $M$ or $b$ or $d$. 

By Menger's theorem, we may assume that $h$ can be separated in $H$ from the proper branch vertices of $M$
by a set $K\subseteq V(H)\sm\{h\}$ of at most two vertices; otherwise we are done. 
In particular, the $3$-fan $F$ has to contain at least one of $b$ and $d$; let 
us say it contains $b$.
Moreover, the $h$--$b$~path $L_b$ in $F$ cannot meet $K$ as $K$ already has to meet the two other paths in the $3$-fan $F$. 
We are done if $b$ is a proper branch vertex itself. Thus we may assume that there is a unique subdivided edge $E$ of $M$ that contains $b$ 
in its interior. One endvertex of $E$ is $a$. The set $K$ also has to separate $b$ from the endvertices of $E$ (as we can reach $b$ from $h$ via $L_b$
without meeting $K$), which implies $K\subseteq V(E)$, and $a\in K$ as $b$ is a neighbour of $a$.
This implies $a \in V(F)$.
Now consider the $h$--$a$~path $L_a$ in $F$, and observe that $L_a$ is internally disjoint from $K$ as $a\in K$. 
Furthermore, since $b \not\in V(L_a)$, the penultimate vertex of $L_a$ is a neighbour $g\neq b$ of $a$.
Then, as $H$ is subcubic, $g$ lies on a subdivided edge $E'$ of $M$ that is not $E$. By extending $hL_ag$ along $E'$ to the endvertex of $E'$ that is not $a$, we obtain a path from $h$ to a proper branch vertex of $M$ that avoids $E$. Since $K\subseteq V(E)$, that path also avoids $K$, a contradiction.
\end{proof}

The next lemma gives us control over $\Phi$, at least for the set $B$.

\begin{lemma} \label{treew:lem:BonBStar}
$\Phi(B) \subseteq B^* \cup V(W)$. 
\end{lemma}

\begin{proof}
Consider a vertex $z \in \Phi(B) \sm V(W)$. 
First observe that, by definition of $B$, every vertex in $\Phi(B)$ has degree at least~$3$ in $Z$. Thus, for $z$ there is a vertex $h$ of $H$ with $z=h^*$.
We will show that $h\in B$, which then implies $z=h^*\in B^*$.

As $h^*\in\Phi(B)$, there is a vertex $g\in B$ with $h^*=\Phi(g)$. 
Since $g\in B$, there is a $3$-fan in $\Phi(H)$ connecting $h^*$ to the set of proper branch vertices of a $10$-wall' $R\subseteq\Phi (H)$. 
We define $O$ to be the union of this fan and $R$. 
If $O$ is disjoint from the proper branch vertices of $M^*$ and also disjoint from $b^*$ and $d^*$, then it is also disjoint from $W^0$ and we can find a corresponding wall' and fan in $H$, implying that $h \in B$.
(When pulling back from $Z$ to $H$, paths between proper branch vertices of $R$ can become shorter, so that the resulting graph in $H$ may be missing some of the required degree~$2$ vertices to be considered a wall; this is precisely the reason why we make do with walls'.)
Therefore, we conclude that $O$ contains some proper branch vertex of $M^*$ (and thus potentially also a part of $W^0$) or that $O$ contains $b^*$ or $d^*$.

Next, suppose that there is no $2$-separator that separates $h^*$ from all proper branch vertices of $M^*$ and from $b^*$ and $d^*$ in $Z$.
By Menger's theorem, there is thus a $3$-fan from $h^*$ to the proper branch vertices of $M^*$ or $b^*$ or $d^*$. We apply Lemma~\ref{treew:lem:bd3fan} to obtain a $3$-fan in $H$ from $h$ to proper branch vertices of $M$ only, which proves $h \in B$.

We conclude that there is a $2$-separator $\{x,y\}\subseteq V(Z-h^*)$ that separates $h^*$ from all proper branch vertices of $M^*$ and all terminals of $W$.
As every vertex of degree~$3$ in $O$ is connected via three internally disjoint paths to $h^*$, we deduce that there is an $x$--$y$~path $P$ in $O$ that contains all vertices that are separated by $\{x,y\}$ from $h^*$ in $O$ and such that all interior vertices of $P$ have degree~$2$ in $O$.
As $O$ contains a proper branch vertex of $M^*$ or a terminal, the  path $xPy$ must contain a vertex from $(V(M))^*$.
Pick $p,q$ to be the first respectively the last vertex of $(V(M))^*$ on $P$, and choose a $p$--$q$~path $Q$ in $M^*$.
 Note that $Q$ is disjoint from $O - pPq$ since $O \cap M^* \subseteq pPq$.
 Moreover, note that $Q$ is disjoint from $W^0$ as $M^*$ is disjoint from $W^0$.
Replacing $pPq$ by $Q$, we obtain a new graph $O'$ that is the union of a $10$-wall' $R'$ with a $3$-fan from $h^*$ to the branch vertices of $R'$ that is disjoint from $W^0$. We then also find in $H$ a $3$-fan from $h$ to the branch vertices of a $10$-wall', which again leads to $h\in B$.
\end{proof}

With the next two lemmas we show that, with only a few exceptions, a vertex in $B$ is mapped to a vertex in $B^*$ under $\Phi$.

\begin{lemma} \label{treew:lem:XeqY}
$|B^* \sm \Phi(B)| = |\Phi(B) \cap (V(W) \sm B^*)|$
\end{lemma}

\begin{proof}
By Lemma~\ref{treew:lem:BonBStar}, we have
\begin{align*}
|\Phi(B) \cap B^*| + |\Phi(B) \cap (V(W) \sm B^*)| & = |\Phi(B)|  = |B| =|B^*| \\
&= |B^* \cap \Phi(B)| + |B^* \sm \Phi(B)|.
\end{align*}
\end{proof}


\begin{lemma} \label{treew:lem:XisSmall}
$|B^* \sm \Phi(B)| \leq 52$.
\end{lemma}
\begin{proof}
By Lemma~\ref{treew:lem:XeqY}, it suffices to show that $|\Phi(B) \cap (V(W) \sm B^*)| \leq 52$.
We show that $|\Phi(B) \cap V(W^0)| \leq 48$, which proves the above claim since $V(W) \sm B^*$ may differ from $V(W^0)$ only in the $4$ terminals of $W$.

Let $z \in \Phi(B) \cap V(W^0)$, and let $h$ be such that $z=\Phi(h)$.
By definition of $B$, $h$ has a $3$-fan to proper branch vertices of a $10$-wall' $M'$ in $H$. 
By Lemma~\ref{nofitlem}, some proper branch vertex of $M'$ needs to be mapped outside $W$ under $\Phi$. 
Then, however, there is a $3$-fan $T_{z}\subseteq\Phi(H)$ from $z$ to a set of vertices in $Z-W$. This $3$-fan must contain at least three terminals of $W$, and thus at least one of $a^*$ and $b^*$. 

Since $z \in V(W^0)$, it lies in one or possibly two blocks of $W - \{a^*, b^*\}$. 
We say that a block $O$ of $W - \{a^*, b^*\}$ \emph{owns} a vertex $z\in \Phi(B) \cap V(W^0)$ if $z$ is incident in $\Phi(H)$ with at least two edges of $O$. As each $z\in\Phi(B)\cap V(W^0)$ has degree~$3$, every vertex in $\Phi(B)\cap V(W^0)$ is owned by exactly one block of $W-\{a^*,b^*\}$. 

Now, assume that the block $O$ owns $z$. If $z$ is not a bottleneck vertex, then the three paths in $T_{z}$ cannot all leave $O$ through its two bottleneck vertices: one such path traverses an edge between $O$ and $a^*$ or $b^*$. 
The same happens if $z$ is a bottleneck vertex: then the two paths in $T_{z}$ with an edge in $O$ cannot both leave $O$ through the remaining bottleneck vertex. Therefore, whenever a block $O$ owns a vertex in $\Phi(B)$, there must be an edge between $O$ and $\{a^*,b^*\}$ in $\Phi(H)$. As $a^*,b^*$ both have degree at most~$3$ in $\Phi(H)$, at most six blocks may own vertices in $\Phi(B)$. 

How many vertices in $\Phi(B)$ may be owned by a block $O$ of $W-\{a^*,b^*\}$?
Every $z \in \Phi(B)\cap V(W^0)$ that is not a bottleneck vertex must have a bottleneck vertex as its neighbour in $\Phi(H)$ since $z$ has degree~$3$, see Figure~\ref{treew:fig:condWall}.
As each bottleneck vertex has degree at most~$3$ in $\Phi(H)$, we conclude that each block contains at~most six non-bottleneck vertices of $\Phi(B)$.
Together with the two bottleneck vertices, we obtain $\leq 8$ vertices of $\Phi(B)$ per block. 
As at most six blocks may own vertices in $\Phi(B)$, we obtain at most $48$ vertices in blocks of $W - \{a^*, b^*\}$. Together with the terminals, this yields $|\Phi(B) \cap (V(W) \sm B^*)| \leq 52$.
\end{proof}

Define $B_M$ to be the set of all vertices in $H$ that send a $3$-fan to proper branch vertices of $M$. 
We note that $B_M$ contains all proper branch vertices of $M$, and $B_M \subseteq B$.

\begin{lemma}\label{BMlem}
Let $h^*$ be a vertex in $Z - W$ with a $3$-fan $T\subseteq Z$ to vertices in $B^*_M$. 
Then $h\in B_M$.
\end{lemma}
\begin{proof}
Suppose there is a set $X$ of at most two vertices that separates $h^*$ from all proper branch vertices of $M^*$ in $Z$. 
Because $X$ cannot separate $h^*$ from all three endvertices of $T$, there exists a path $P$ in $Z-X$ between $h^*$ and some vertex $g^*\in B_M^*$.
As there is, by definition, a $3$-fan from $g$ to proper branch vertices of $M$ in $H$, there is also $3$-fan from $g^*$ to proper branch vertices of $M^*$ in $Z$, and then, as $|X|\leq 2$, also a path $Q$ from $g^*$ to a proper branch vertex of $M^*$ in $Z-X$. However, $P\cup Q$ is disjoint from $X$ but contains a path from $h^*$ to a proper branch vertex of $M^*$, which is impossible.

Therefore, by Menger's theorem, there is a $3$-fan $T'$ from $h^*$ to proper branch vertices of $M^*$. By Lemma~\ref{treew:lem:bd3fan}, we obtain $h\in B_M$.
\end{proof}

In conjunction with Lemma~\ref{treew:lem:XisSmall}, the next lemma will be used to repair $M$, that is to prove that $\Phi(H)$ contains most proper branch vertices of $M^*$ and sufficient subdivided edges in between them.

\begin{lemma} \label{treew:lem:pathsBetweenBranchverticesInMStar} 
Let $g,h \in B_M$, and let $L$ be a $g$--$h$~path in $H - e_1 - e_2$. Let $P$ be a $g^*$--$h^*$~path in $Z$ such that 
$V(P) \cap (V(H))^* = (V(L))^*$ and such that $P$ is disjoint from $B^*\sm \Phi(B)$.
For every vertex $i^* \in V(P)$ that is a terminal, we furthermore require that $i^*$ has degree~$2$ in $\Phi(H) - W^0$. 
Let $\mathcal S$ be the set of all $B^*_M$-paths in $Z$ that 
are disjoint from the interior of  $P$ and that have at most one endvertex with $P$ in common.
Then there is a $g^*$--$h^*$~path $Q$ in $\Phi(H) - W^0$ that is internally disjoint from every path in $\mathcal{S}$.
\end{lemma}

\begin{proof}
We do induction on the number $n$ of internal vertices of $P$ that lie in $B^*_M$. Because it is shorter, we start with the induction step. Thus, assume that $n>0$, ie that $P$ contains an internal vertex $k^*\in B^*_M$.
We split the path $P$ into $P_1=g^*Pk^*$ and $P_2=k^*Ph^*$, and observe that both paths have fewer than $n$ internal vertices in $B^*_M$. As subpaths of $P$, the paths $P_1$ and $P_2$ still satisfy the conditions of the lemma.
Now induction yields a $g^*$--$k^*$~path $Q_1\subseteq\Phi(H)-W^0$ and a $k^*$--$h^*$~path $Q_2\subseteq \Phi(H)-W^0$.
Let $Q$ be a $g^*$--$h^*$~path contained in $Q_1\cup Q_2\subseteq\Phi(H)-W^0$.  

Consider a path $S\in\mathcal S$, and suppose that $S$ meets $Q$ in an internal vertex of $Q$.
We first note that $S$ cannot contain $k^*$ as any path in $\mathcal S$ is disjoint from the interior of $P$. Thus, $S$ meets an internal vertex of $Q_1$ or of $Q_2$, say of $Q_1$. This, however, is impossible as $S$ is disjoint from the interior of $P_1$, and may have at most one endvertex with $P_1$ in common. Therefore, $Q$ is as desired, and we have proved the induction step. 

It remains to establish the induction start. Then, $n=0$, which implies that:
\begin{equation}\label{nointernal}
\emtext{No internal vertex of $P$ lies in $B^*_M$.}
\end{equation}

As $P$ is disjoint from $B^* \sm \Phi(B)$, we get $g^* \in \Phi(B)$. 
Thus, there is a $10$-wall' $R$ and a $3$-fan from $g^*$ to proper branch vertices of $R$ in $\Phi(H)$. 
We denote by $O$ the union of $R$ and this $3$-fan. Note that $O$ is a subgraph of $\Phi(H)$.

Let us prove that:
\begin{equation}\label{nbs}
\text{\em 
For any neighbour $g_0$ of $g$ in $H-e_1-e_2$, the vertex $g_0^*$ lies in $O$.
}
\end{equation}
Indeed, since $H$ is subcubic and since $g^*$ has degree~$3$ in $O$ it follows that for  every neighbour $g_0$ of $g$ in $H$, we have $g^*_0\in V(O)$ --- unless $g^*$ is a terminal. Then, since $g \in B_M$, $g$ has degree~$2$ in $H-e_1-e_2$, and by assumption, $g^*$ has degree~$2$ in $\Phi(H)-W^0$: again, for every neighbour $g_0$ of $g$ in $H-e_1-e_2$, the vertex $g_0^*$ lies in $O$.


Let $g_1$ be the neighbour of $g$ in $H-e_1-e_2$ that lies in $L$, the $g$--$h$~path in $H$. 
It now follows from~\eqref{nbs} that:
\begin{equation}\label{v1v2v3}
\begin{minipage}[c]{0.8\textwidth}\em
For the neighbour $g_1$ of $g$ in $L$ it holds that $g_1^*\in V(O\cap P)$.
\end{minipage}\ignorespacesafterend 
\end{equation} 

Among all vertices in $O\cap P$, pick $k^*$ to be closest to $h^*$ on $P$.
Note that since $g_1^* \in V(P)$, it is a candidate for $k^*$. Thus we immediately have $k^*\neq g^*$.

Next, we claim:
\begin{equation}\label{v1v2v3B}
k^*\in B^*_M
\end{equation} 
Suppose not. In particular, $k^*\neq h^*$ as $h^*\in B^*_M$.
By Lemma~\ref{BMlem}, there are two vertices $x_1,x_2\neq k^*$ that
separate $k^*$ from $B^*_M$. As $k^*Ph^*$ is a $k^*$--$B^*_M$~path, one
of $x_1,x_2$ lies in  $k^*Ph^*$, say $x_1$. By choice of $k^*$, the subpath $k^*Ph^*$ meets $O$ only in $k^*$, which implies that $x_1\not\in V(O)$.
In $O$ there are two internally disjoint $k^*$--$g^*$~paths. Since $x_1\not\in V(O)$, one of the 
two internally disjoint $k^*$--$g^*$~paths in $O$ is disjoint from $x_1,x_2$ unless $x_2 = g^*$. We thus conclude $x_2 = g^*$. 

Next, as $g^*\in B^*_M$ by assumption, it follows that there exists a $3$-fan $T$ from $g^*$ to $B^*_M$ in $Z$. 
Since $H$ is subcubic, for two neighbours $g_1, g_2$ of $g$ in $H - e_1 -e_2$, $g_1^*$ and $g_2^*$ lie on different paths of the $3$-fan $T$ from $g^*$ to $B^*_M$ in $Z$. That is, there are disjoint paths $P_1,P_2$, where $P_1$ is a $g_1^*$--$B_M^*$~path and $P_2$ is a  $g_2^*$--$B_M^*$~path, both disjoint from $x_2 = g^*$. As $O$ is $2$-connected, there are paths in $O$ from $k^*$ to $g_1^*$ and $g_2^*$ that avoid $\{x_1, x_2\}$ (recall that $x_1 \not\in V(O)$). Since $P_1$ and $P_2$ are disjoint, at least one of them is disjoint from $x_1$. Thus there is a $k^*$--$B_M^*$~path in $Z - \{x_1, x_2\}$, a contradiction. This proves~\eqref{v1v2v3B}.

With~\eqref{nointernal} we get that $k^*=h^*$, which implies $h^*\in V(O)$. 
We claim that:
\begin{equation}\label{progress}
\emtext{
There is a $g^*$--$h^*$~path $Q$ in $O$ whose second vertex in $(V(H))^*$ is $g_1^*$.
}
\end{equation}

Since $O$ is $2$-connected and $g_1^* \in V(O)$ by \eqref{v1v2v3}, there is a $g_1^*$--$h^*$~path $Q'$ in $O$ that is disjoint from $g^*$. Since $g_1$ is a neighbour of $g$ in $H - e_1 - e_2$, there is a $g^*$--$g_1^*$~path $Q''$ of length~$2$ in $O$, which by construction of $Z$ is internally disjoint from $Q'$. Combining those to $Q = Q'' \cup Q'$ thus yields the desired $g^*$--$h^*$~path $Q$ in $O$. This proves \eqref{progress}.

Note that $Q\subseteq\Phi(H)$. Thus, to finish the proof we need to show that $Q$ is disjoint from $W^0$; and that $Q$ is internally disjoint from every $S\in\mathcal S$.

Suppose  
that the interior of $Q$ meets either $W^0$ or some path in $\mathcal S$, and let $q$ be the first vertex in the interior of $Q$ where that happens. 
Next, among all vertices in $g^*Qq\cap P$, pick $\ell^*$ to be the one closest to $h^*$ on $P$. 

\begin{figure}[htb]
\centering
\begin{tikzpicture}
\tikzstyle{bvx}=[thick,circle,inner sep=0.cm, minimum size=1.5mm, fill=white, draw=black,fill=black]
\tikzstyle{vx}=[thick,circle,inner sep=0.cm, minimum size=1.5mm, fill=white, draw=black]
\tikzstyle{Qpath}=[ultra thick, dunkelgrau]

\node[bvx,label=left:$g^*$] (v) at (0,0){};
\node[bvx,label=right:$h^*$] (w) at (4.7,0){};

\draw[hedge] (v) to (w);

\node at (4,-0.3) {$P$};

\draw[Qpath,rounded corners=5pt] (v) to [rounded corners=0pt] (0.7,0) to  (0.7,-0.3) to
(1.3,-0.8) to (2,-0.9) to node[auto,swap,text=black]{$Q$} (2.4,-0.5) to (3,-0.2) to (3,0);

\draw[Qpath] (3,0) to (2.7,0) arc [start angle=0, end angle=180, radius=0.5]
to ++(-0.4,0)
arc [start angle=180, end angle=90+30, radius=1.3] node[vx,label={[text=black]above right:$q$}] (q) {};

\node[vx,label=above right:$\ell^*$] (u) at (3,0){};

\node[bvx] (s1) at (0.4,1.3){};
\node[bvx] (s2) at (4.2,1.4) {};

\draw[hedge] (s1) to [out=0,in=160] node[midway,auto] {$S$} (q);
\draw[hedge] (q) to [out=-20,in=200] (s2);

\end{tikzpicture}
\caption{Situation in Lemma~\ref{treew:lem:pathsBetweenBranchverticesInMStar}. Vertices in $B^*_M$ in black.}\label{repairfig}
\end{figure}
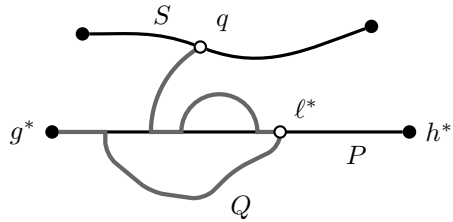

We observe that $\ell^*$ must be an internal vertex of $P$. 
Indeed, $\ell^*\neq h^*$ as $q$ is an internal vertex of $Q$, 
and $\ell^*\neq g^*$ by~\eqref{progress}.
From~\eqref{nointernal} it follows that $\ell^*\notin B^*_M$, and from 
 Lemma~\ref{BMlem} it follows that there is a set $Y=\{y_1,y_2\}$ of at most two vertices 
that separates $\ell^*$ from $B^*_M$ in $Z$.

As the paths $g^*Q\ell^*$ and $\ell^*Ph^*$ meet only in $\ell^*$ by choice of $\ell^*$, it follows that one vertex in $Y$, $y_1$ say, lies in $\ell^*Ph^*$ and the other, $y_2$, in $g^*Q\ell^*$. Now, the path $\ell^*Qq$ meets $g^*Q\ell^*$ and $\ell^*Ph^*$ also only in $\ell^*$ and thus is disjoint from $Y$. As a consequence, $q$ cannot lie in $W^0$ as every vertex in $W^0$ sends a $3$-fan to $B^*_M$. 
Therefore, $q$ lies on a path $S\in\mathcal S$.

Note that 
as $\ell^*Qq$ is disjoint from $Y$ and as the endvertices of $S$ lie in $B^*_M$, it follows that 
both vertices in $Y$ must lie on $S$.
If $y_1$ lies in $S$ then, as $y_1\in V(P)$ and as $P$ is internally disjoint from $S$, the vertex $y_1$ must
be an endvertex of $P$, ie, $y_1=h^*$. As $S$ is a $B^*_M$-path, it follows that $y_1$ is an endvertex of $S$.  
That $y_2\in V(g^*Q\ell^*)$ lies in $S$ implies, too, that $y_2$ must be an endvertex of $S$: Indeed, $q$ was the first internal vertex on $Q$ to lie in $S$, and thus $y_2=g^*$, which lies in $B^*_M$. But now, $S$ has both endvertices with $P$ in common, which is not allowed for a path in $\mathcal S$.
We have obtained the final contradiction that proves the lemma.
\end{proof}

We are done if we find an ($a^*$--$b^*$, $c^*$--$d^*$) linkage in $\Phi(H)\cap W$. 
The next lemma tells us that if there is no such linkage then we obtain two different paths between the 
terminals, one inside  the Heinlein wall and one outside.

\begin{lemma} \label{treew:lem:ab-cd-pathInWWithEndsInSameComponent}
Either there is an ($a^*$--$b^*$, $c^*$--$d^*$) linkage in $\Phi(H)\cap W$, or 
there is an $\{a^*, b^*\}$--$\{c^*, d^*\}$~path in $\Phi(H) \cap W$ whose endvertices are in the same component of $\Phi(H) - W^0$. 
\end{lemma}

\begin{proof}
We proceed by case distinction. First, consider the case that there is a $v\in\Phi(B)$ such that $v$
lies in $W^0$ or such that $v$ is a terminal with degree at least~$2$ in $\Phi(H)\cap W$.

As $v\in\Phi(B)$, there is a $3$-fan $T$ in $\Phi(H)$ from $v$ to the proper branch vertices of some $10$-wall' $R\subseteq\Phi(H)$.
Moreover, as $R$ is too large to fit into $W$ by Lemma~\ref{nofitlem}, there must be some proper branch vertex $w$ of $R$ outside $W$. 
Thus, $T\cup R$ contains three internally disjoint $v$--$w$~paths $P_1,P_2,P_3 \subseteq \Phi(H)$. 

By definition of $v$, there are three terminals that are incident with an edge in $(P_1 \cup P_2 \cup P_3) - W^0$. 
Therefore, $P_1\cup P_2\cup P_3$ contains an $\{a^*, b^*\}$--$\{c^*, d^*\}$~path that lies in $\Phi(H)\cap W$. Moreover, the endvertices of that path are connected in $\Phi(H)-W^0$ via $P_1\cup P_2\cup P_3-W^0$.

Second, we consider the case when $\Phi(B)\cap V(W^0)=\emptyset$ and when every terminal in $\Phi(B)$ has degree at most~$1$ in $\Phi(H)\cap W$. We claim that

\begin{equation}\label{noNewBvertex}
\Phi(B)\cap (V(W) \sm B^*)=\emptyset
\end{equation} 

Since $\Phi(B)\cap V(W^0)=\emptyset$ and since $\{a^*,c^*\} \subseteq B^*$, the claim~\eqref{noNewBvertex} can only be violated if $b^*\in\Phi(B)\sm B^*$ or if $d^*\in\Phi(B)\sm B^*$. While $b^*$ and $d^*$ are not exchangeable, they are largely symmetric for the purpose of the proof of~\eqref{noNewBvertex}. Therefore, we only concentrate on $b^*$ and consider the case that $b^*\in\Phi(B)$ and then show that this implies $b^*\in B^*$. The proof for $d^*$ is similar.

From  $b^*\in\Phi(B)$ it follows that there are three paths $P_1, P_2, P_3$ in $\Phi(H)$ from $b^*$ to proper branch vertices of some $10$-wall' $R \subseteq \Phi(H)$ such that $P_1, P_2, P_3$ are disjoint except for $b^*$. 
Note that all proper branch vertices of $R$ lie in $Z - W^0$ as $\Phi(B)$ is disjoint from $W^0$.
Therefore, $R$ may only intersect $W^0$ in at most two paths. (Here, we also use that every terminal in $\Phi(B)$ has degree at most~$1$ in $\Phi(H)\cap W$.) Let $Q_1, Q_2$ be the paths in $R$ between proper branch vertices of $R$ that are incident with $W^0$ (if they exist at all).

Let $P \in \{P_1, P_2, P_3, Q_1, Q_2\}$, and observe that $P\subseteq\Phi(H)$. 
As the endvertices of $P$ are  either proper branch vertices of $R$ or $b^*$, it follows that they lie in $V(H)^*$. We denote them by $g^*$ and $h^*$. Moreover, as we assume $b^*\in\Phi(B)$ it follows that $\Phi(H)$ contains three internally disjoint disjoint $g^*$--$h^*$~paths. Only two of these may intersect $W^0$. As a consequence, the endvertices $g^*$ and $h^*$ are contained in the same component of  $\Phi(H) - W^0$.
Therefore, if $P \cap W$ contains exactly one non-trivial $\{a^*, b^*\}$--$\{c^*, d^*\}$~path $Q$, then, with the help of $P-W^0$, we see that the endvertices of $Q$ are in the same component of $\Phi(H) - W^0$. As, moreover, $P\cap W$ contains a path between the endvertices of $Q$, we have found a path as in the statement of the lemma and are done. 

If, on the other hand, $P \cap W$ contains two non-trivial $\{a^*, b^*\}$--$\{c^*, d^*\}$~paths, we can use $e_1$ and $e_2$ to find a $g$--$h$~path $I_P$ in $H$ with $V(I_P)^*\subseteq V(P)$.
If $P \cap W$ contains an $a^*$--$b^*$~path or a $c^*$--$d^*$~path (or both), we can again use $e_1$ or $e_2$ to find a $g$--$h$~path $I_P$ in $H$ with $V(I_P)^*\subseteq V(P)$.
(If $P \cap W$ contains both an $a^*$--$b^*$~path and a $c^*$--$d^*$~path, we can actually stop as then we have the desired linkage.)
Finally, if $P \cap W$ contains no non-trivial path then, too, we easily find a $g$--$h$~path $I_P$ in $H$ with $V(I_P)^*= V(P)\cap V(H)^*$.

Since $R$ intersects $W^0$ only in $Q_1$ and $Q_2$ (if these exist at all), using $I_{Q_1}$ and $I_{Q_2}$, we find a $10$-wall' $M'$ in $H$ such that $V(M')^*\subseteq V(R)$. 
In the same way, we note that the $b$--$M'$~paths $I_{P_1},I_{P_2},I_{P_3}$ in $H$ satisfy $V(I_{P_j})^*\subseteq V(P_j)$ for $j=1,2,3$. In particular, $I_{P_1},I_{P_2},I_{P_3}$ are pairwise disjoint except for $b$. In total, we have found a $3$-fan from $b$ to a $10$-wall', which implies that $b\in B$ and thus $b^*\in B^*$. This proves \eqref{noNewBvertex}.

By Lemma~\ref{treew:lem:XeqY}, it follows from \eqref{noNewBvertex} and $|B|=|\Phi(B)|$ that 
\begin{equation}\label{BphiB}
B^*=\Phi(B).
\end{equation}

In particular, the terminals $a^*$ and $c^*$ lie in $\Phi(B)$, which implies that there is, for every terminal~$v \in \{a^*,c^*\}$, a $3$-fan $T\subseteq\Phi(H)$ from $v$ to proper branch vertices of some $10$-wall' $R\subseteq\Phi(H)$. 
Note that all proper branch vertices of $R$ lie in $\Phi(B)$ and thus outside $W^0$. Therefore, there is for every  $v \in \{a^*,c^*\}$ a path $Q_v$ that starts in $v$, that ends in another terminal and that is completely contained in $W$. Moreover, 
via the $3$-fan $T$, there is a path between the endvertices in $\Phi(H) - W^0$. (Observe that the paths $Q_{a^*},Q_{c^*}$ do not have to be disjoint, nor distinct.)

If $Q_{a^*}$ ends in $\{c^*,d^*\}$ or if $Q_{c^*}$ ends in $\{a^*,b^*\}$, we observe that $Q_{a^*}$ or $Q_{c^*}$ is an $\{a^*, b^*\}$--$\{c^*, d^*\}$~path in $\Phi(H) \cap W$ whose endvertices are in the same component of $\Phi(H) - W^0$ and we are done. Thus we may assume that $Q_{a^*}$ is an $a^*$--$b^*$~path and $Q_{c^*}$ is a $c^*$--$d^*$~path.
If $Q_{a^*}$ is disjoint from $Q_{c^*}$, they form an $(a^*$--$b^*$, $c^*$--$d^*$) linkage in $\Phi(H)\cap W$, which was what we wanted. 
Thus, we may assume that $Q_{a^*}$ intersects $Q_{c^*}$, which implies that $Q_{a^*}\cup Q_{c^*}$ contains an $a^*$--$c^*$~path $P$. We then apply Lemma~\ref{treew:lem:pathsBetweenBranchverticesInMStar} to $a^*,c^*$ in the role of $g^*,h^*$, and to some $a$--$c$~path $L$ in $M-e_1-e_2$. Note that $(V(L))^*$ is automatically disjoint from $B^*\sm\Phi(B)$, as the latter set is empty, by~\eqref{BphiB}.
The path we obtain from the lemma then shows that the endvertices of $P$ lie in the same component of $\Phi(H)-W^0$, and we are done.
\end{proof}

In the next lemma we will use planarity arguments. To this end, if $G$ is a planar graph that is drawn in the plane, ie, if $G\subseteq\mathbb R^2$, then we define the \emph{interior} $\intr(G)$ as the set $\mathbb R^2\sm F$, where $F$ is the outer face (the unbounded face) of $G$.

We have reached the final lemma that concludes the proof of Theorem~\ref{treew:theo:theoremAsWallSize}. 
\begin{lemma*linkage}
$\Phi(H)$ contains an ($a^*$--$b^*$, $c^*$--$d^*$) linkage in $W$.
\end{lemma*linkage}

\begin{proof}
Suppose that $\Phi(H)\cap W$ does not contain any ($a^*$--$b^*$, $c^*$--$d^*$) linkage.
Then Lemma~\ref{treew:lem:ab-cd-pathInWWithEndsInSameComponent} yields $v_1\in \{a^*, b^*\}$, $v_2\in \{c^*, d^*\}$
and  $v_1$--$v_2$~paths $P$ and $Q$ such that $P\subseteq \Phi(H) \cap W$ and $Q\subseteq \Phi(H) - W^0$.

Set $D=\{h \in B : h^* \not\in \Phi(B)\}$ and observe that Lemma~\ref{treew:lem:XisSmall} implies that $|D|\leq 52$. 
Every vertex is incident with at most one row and one column of the wall' $M$. Thus, there is a wall' $M'\subseteq M-D-\{a,b,c,d\}$ that contains all but at most $56$ rows and columns of $M$, and  that is disjoint from $D$ and from the terminals $a,b,c,d$. 
We write ${M'}^*\subseteq Z$ for the subwall' of $M^*$ that contains all images of the branch vertices of $M'$ under~$*$.

As all proper branch vertices of $M'$ are in $B_M$ and as ${M'}^*$ is disjoint from $B^*\sm\Phi(B)$ we can  apply Lemma~\ref{treew:lem:pathsBetweenBranchverticesInMStar} to every (subdivided) edge $gh$ of $M'$ to see that there is a $g^*$--$h^*$~path in $\Phi(H)-W^0$.
Moreover, as such a path is a $B^*_M$-path (and thus in $\mathcal{S}$ with respect to the lemma), the obtained paths are all internally disjoint.
Replacing the subdivided edges of ${M'}^*$ one by one in this way, we obtain a wall' $R$ in $\Phi(H)-W^0$ whose proper branch vertices are identical with those from ${M'}^*$. 
In particular, for every row (resp.\ for every column) of ${M'}^*$ there is a row (resp.\ a column) of $R$ with the same proper branch vertices.

We note for later that
\begin{equation}\label{Rphi}
R\subseteq  \Phi(H)-W
\end{equation}
We make a second observation. 
The graph $Z-W^0$ is planar as $H$ is planar, and, in what follows, we consider a fixed drawing of $Z-W^0$. 
Then, the interior $\intr(S)$ of any brick $S$ of $M^*$ is well-defined.
We may assume that $Z-W^0$ is drawn in such a way that no brick interior contains the outercycle of $M^*$. 
We use this to observe that if $S'$ is a brick of ${M'}^*$ and if $S$ is the corresponding brick of $R$ with the same proper branch vertices  then any vertex in $\intr(S')$ lies  in the interior $\intr(S)$ or in the interior of a brick of $R$ that is adjacent to $S$, ie, that shares a subdivided edge with $S$.

Recall the $v_1$--$v_2$~path $Q$ contained in $\Phi(H) - W^0$.
We claim:
\begin{equation}\label{Rapart}
\begin{minipage}[c]{0.8\textwidth}\em
$Q$ meets $R$, and if $q_1$ is its first and $q_2$ its last vertex in $R$ then $q_1,q_2$ are $8$-apart in $R$.
\end{minipage}\ignorespacesafterend 
\end{equation} 
As each pair of one vertex from $\{a,b\}$ and one of $\{c,d\}$ is $70$-apart in $M$, it follows that $v_1,v_2$ are $70$-apart in $M^*$. (Recall that $M^*$ is a subdivision of $M-e_1-e_2$ in $Z-W^0$.) As every path in $M^*$ from $v_1$ or from $v_2$ to the outercycle of $M^*$ meets at least $70$ rows or columns, and as ${M'}^*$ contains all but $56$ rows and all but $56$ columns of $M^*$ it follows that there are bricks $S'_1,S'_2$ of  ${M'}^*$ such that $v_i\in\intr(S'_i)$ for $i=1,2$.

Consider a path $Q'\subseteq {M'}^*$ from a vertex of $S'_1$ to a vertex of $S'_2$ and suppose that $Q'$ meets ${M'}^*$ fewer than 10 times. Then follow $Q$, which is a path in $Z-W^0$, from $v_1$ to the first vertex in $S'_1$,
then along $S'_1$ to the first vertex of $Q'$, then along $Q'$ to $S'_2$, from there to the last vertex of $Q$ in $S'_2$ and along $Q$ to $v_2$. The resulting $v_1$--$v_2$ path $Q''\subseteq Z-W^0$ meets fewer than $14$ rows and columns of ${M'}^*$ (each of the bricks $S'_1$ and $S'_2$ may contribute at most two more rows and columns). As ${M'}^*$ contains all but $56$ rows and columns of $M^*$ we see that $Q''$ meets fewer than $70$ rows and columns of $M^*$, which is impossible as $v_1,v_2$ are $70$-apart in $M^*$.  
In a similar way, we see that each path from $v_1$ or from $v_2$ to the outercycle of ${M'}^*$ meets $10$ rows or columns of ${M'}^*$.
Therefore, $S'_1,S'_2$ are $10$-apart in ${M'}^*$.

As we had observed that the interior of each brick of $R$ is contained in the interior of the corresponding brick in ${M'}^*$
together with the interiors of adjacent bricks, it follows that there are bricks $S_1$ and $S_2$ of $R$ such that $v_i\in\intr(S_i)$ for $i=1,2$ and such that $S_1,S_2$ are $8$-apart in $R$. 

As a consequence, the path $Q$, which is entirely contained in the plane graph $Z-W^0$, meets $R$ (in at least eight vertices). Denote by $q_1$ the first vertex of $Q$ in $R$, and let $q_2$ be the last vertex of $Q$ in $R$. Then $q_1$ lies in the brick $S_1$, and $q_2$ lies in $S_2$. Therefore, $q_1,q_2$ are $8$-apart in $R$. This proves~\eqref{Rapart}. 

Recall the $v_1$--$v_2$~path $P$ contained in $\Phi(H)\cap W$. As $H$ is planar, and as as $Q\cup P\cup R\subseteq \Phi(H)$, it follows that $Q\cup P\cup R$ is planar, too. Consider $q_1Qv_1\cup P \cup v_2Qq_2$: this is a $q_1$--$q_2$~path that meets the wall' $R$ only in its endvertices since $P\subseteq W$, while $R$ is disjoint from $W$, by~\eqref{Rphi}.
However, $q_1,q_2$ are $8$-apart, by~\eqref{Rapart}. Clearly, this is impossible in a planar graph. The final contradiction proves the lemma.
\end{proof}

\bibliographystyle{amsplain}

\bibliography{chapters/literature}{}

\end{document}